\documentclass[a4paper,oneside,12pt]{amsart}

\pdfoutput = 1 % sets the output file to be in PDF format (instead of DVI)

\usepackage[
a4paper ,
top    = 3cm ,
bottom = 2.5cm ,
left   = 2.8cm ,
right  = 2cm ,
]{geometry}

\usepackage[T1]{fontenc} % allows to type accented characters directly into the .tex file

\usepackage[american]{babel}
\usepackage{mathtools} % amsmath improvement
\mathtoolsset{mathic} % math italics correction

\usepackage[lf]{OldStandard}
\usepackage[scale=0.85]{FiraSans} % sans serif text font
\usepackage[scale=1]{inconsolata} % sans serif monospaced ("typewriter") font

\usepackage{eulerpx} % math font
\usepackage{eucal}

\usepackage{bm} % improved bold math
\usepackage{xcolor}

\makeatletter
\g@addto@macro\bfseries{\boldmath}
\makeatother

\allowdisplaybreaks[3] % allow page breaks in multiline equations

\makeatletter
\renewcommand\operator@font{\sf}
\makeatother

\DeclareTextFontCommand{\noun}{\scshape} % small caps for proper nouns

\newcommand{\C}{ \varmathbb{C} }
\newcommand{\D}{ \varmathbb{D} }
\newcommand{\E}{ \varmathbb{E} }
\newcommand{\T}{ \varmathbb{T} }
\newcommand{\Z}{ \varmathbb{Z} }

\newcommand{\CC}{ \mathcal{C} }
\newcommand{\CP}{ \mathcal{P} }

\newcommand{\BH}{q}
\newcommand{\Znonneg}{ \varmathbb{Z}_{ \geq 0 } }
\newcommand{\MW}{M}
\newcommand{\EtP}{ E_{ t , P } }
\newcommand{\AsigbeP}{ A_{ \sigma , \beta , P } }
\newcommand{\QsigbeP}{ Q_{ \sigma , \beta , P } }
\newcommand{\Ssigbe}{ S_{ \sigma , \beta } }

\newcommand\mystrut{\rule{0pt}{15pt}}

\DeclareMathOperator{\supp}{supp}

\DeclarePairedDelimiter{\abs}{\lvert}{\rvert}
\DeclarePairedDelimiter{\norm}{\lVert}{\rVert}

\usepackage[normalem]{ulem}
\newcommand{\soutc}{\bgroup\markoverwith{\textcolor{cyan}{\rule[0.4ex]{1pt}{1.7pt}}}\ULon}

\NewCommandCopy\oldsqrt\sqrt
\RenewDocumentCommand{\sqrt}{om}{%
  \IfValueTF{#1}{\mathpalette{\DHLhksqrt}{{\oldsqrt[{#1}]{#2\mskip2mu}}{\oldsqrt{#2\mskip2mu}}}}%
                {\mathpalette{\DHLhksqrt}{{\oldsqrt{#2\mskip1.5mu}}{\oldsqrt{#2\mskip1.5mu}}}}%
}%
\makeatletter
\newcommand*\bold@name{bold}
\newcommand\DHLhksqrt[2]{%
  \begingroup
  \setbox0 =\hbox{$#1\@secondoftwo#2$}%
  \@firstoftwo#2%
  \vrule height \dimexpr \ht0 - %
          \fontdimen8 \ifx#1\scriptscriptstyle\scriptscriptfont\else
                        \ifx#1\scriptstyle\scriptfont\else\textfont\fi
                      \fi3 \relax
         depth -\dimexpr 0.8\ht0 - %
          \fontdimen8 \ifx#1\scriptscriptstyle\scriptscriptfont\else
                        \ifx#1\scriptstyle\scriptfont\else\textfont\fi
                      \fi3 \relax
         width \dimexpr \fontdimen8 \ifx#1\scriptscriptstyle\scriptscriptfont\else
                        \ifx#1\scriptstyle\scriptfont\else\textfont\fi
                      \fi3 \relax
  \endgroup
}%
\makeatother % root symbol with hook

\numberwithin{equation}{section} % numbering of equations includes section number

\usepackage[pagebackref]{hyperref}
\hypersetup{
allcolors     = blue ,
breaklinks    = true ,
colorlinks    = true ,
pdfencoding   = unicode ,
pdfpagelayout = OneColumn ,
pdfstartview  = {FitH} ,
}

\usepackage{bookmark} % PDF bookmarks

\usepackage[
alphabetic ,
msc-links ,
nobysame ,
]{amsrefs}

\makeatletter
\def\print@backrefs#1{\space\SentenceSpace[cited on page(s) \csname br@#1\endcsname]}
\makeatother

\BibSpec{book}{
     + {}  {\PrintPrimary}                {transition}
     + {,} { \textit}                     {title}
     + {.} { }                            {part}
     + {:} { \textit}                     {subtitle}
     + {,} { \PrintEdition}               {edition}
     + {}  { \PrintEditorsB}              {editor}
     + {,} { \PrintTranslatorsC}          {translator}
     + {,} { \PrintContributions}         {contribution}
     + {,} { }                            {series}
     + {,} { \voltext}                    {volume}
     + {,} { }                            {publisher}
     + {,} { }                            {organization}
     + {,} { }                            {address}
     + {,} { \PrintDateB}                 {date}
     + {,} { }                            {status}
     + {,} { \PrintDOI}                   {doi}
     + {}  { \parenthesize}               {language}
     + {}  { \PrintTranslation}           {translation}
     + {;} { \PrintReprint}               {reprint}
     + {.} { }                            {note}
     + {.} {}                             {transition}
     + {}  {\SentenceSpace \PrintReviews} {review}
}

\usepackage[doipre={}]{uri} % DOI hyperlinks for bibliography items

\usepackage{enumitem} % customized lists

\usepackage[framemethod = tikz]{mdframed}
\mdfdefinestyle{barra}{
backgroundcolor   = white ,
linewidth         = 3pt ,
topline           = false ,
bottomline        = false ,
rightline         = false ,
innertopmargin    = -6pt ,
innerbottommargin = 0pt ,
innerrightmargin  = 0pt ,
innerleftmargin   = 5pt ,
}

\theoremstyle{plain}
\newmdtheoremenv[ default , style = barra ]{theorem}[dummy]{Theorem}
\newmdtheoremenv[ default , style = barra ]{corollary}[dummy]{Corollary}
\newmdtheoremenv[ default , style = barra ]{proposition}[dummy]{Proposition}
\newmdtheoremenv[ default , style = barra ]{lemma}[dummy]{Lemma}

\theoremstyle{definition}
\newmdtheoremenv[ default , style = barra ]{definition}[dummy]{Definition}
\newmdtheoremenv[ default , style = barra ]{example}[dummy]{Example}
\newmdtheoremenv[ default , style = barra ]{remark}[dummy]{Remark}
\newmdtheoremenv[ default , style = barra ]{problem}{Problem}

\begin{document}

\title[Asymptotic contractivity of restricted BH constants]{Asymptotic Contractivity of Bohnenblust--Hille Constants with Bounded Monomial Support}

\author[N.~Caro-Montoya]{Nicolás Caro-Montoya}
\address{Departamento de Matemática \\ Centro de Ciências Exatas e da Natureza \\ Universidade Federal de Pernambuco \\ Avenida Jornalista Aníbal Fernandes, S/N - Cidade Universitária \\ Recife/PE - Brasil - 50740-560}
\email[N.~Caro-Montoya]{jorge.caro@ufpe.br}

\author[D.~Núñez-Alarcón]{Daniel Núñez-Alarcón}
\address{Departamento de Matemáticas \\ Universidad Nacional de Colombia \\ Bogotá, Colombia - 111321}
\email[D.~Núñez-Alarcón]{dnuneza@unal.edu.co}

\author[D.~Serrano-Rodríguez]{Diana Serrano-Rodríguez}
\email[D.~Serrano-Rodríguez]{diserranor@unal.edu.co}


\begin{abstract}

We study the optimal \noun{Bohnenblust--Hille} constants \( K_{ m , \MW } \) for complex \( m \)-homo\-geneous polynomials in any number of variables whose monomials with nonzero coefficient involve at most \( \MW \) variables. For every fixed \( \MW \), we prove that these constants satisfy
\[
K_{ m , \MW } \leq A_{ \MW }^{ \MW / m } \cdot m^{ ( \MW^2 - 1 ) / ( 2m ) } \, , \quad \textnormal{for} \ m \geq \MW \, ,
\]
where \( A_{ \MW } \geq 1 \) depends only on \( \MW \). In particular,
\[
K_{ m , \MW } \to 1 \quad \textnormal{as} \ m \to \infty \, ,
\]
and so the corresponding \noun{Bohnenblust--Hille} constants are asymptotically contractive. The proof exploits the homogeneous structure through a decomposition according to exact monomial-support levels, partitions of the set of variables, multilinear \noun{Bohnenblust--Hille} estimates, and interpolation with \noun{Parseval}'s identity.

\end{abstract}

\maketitle

\subsection*{Notation and conventions}

\noindent Throughout this paper \( \Z^+ \) (resp.~\( \Znonneg \)) denotes the set of positive (resp.~nonnegative) integers, and \( \D \) denotes the open unit disc in the field \( \C \) of complex numbers. Given \( n \in \Z^+ \) and \( \alpha = ( \alpha_1 , \dotsc , \alpha_n ) \in \Znonneg^n \)\,, we define \( \abs{ \alpha } \coloneq \alpha_1 + \dotsb + \alpha_n \) and \( \supp ( \alpha ) \coloneq \{ j \colon \alpha_j > 0 \} \), and we denote the cardinality of \( \supp ( \alpha ) \) by \( w( \alpha ) \). For \( z = ( z_1 , \dotsc , z_n ) \in \C^n \), the quantity \( z^{ \alpha } \coloneq z_1^{ \alpha_1 } \dotsm z_n^{ \alpha_n } \) is called a \emph{monomial} (in the variables \( z_1 , \dotsc , z_n \)), and the value \( \abs{ \alpha } \) is the \emph{degree} of \( z^{ \alpha } \). Notice that \( j \in \supp ( \alpha ) \) precisely when the variable \( z_j \) divides the monomial \( z^{ \alpha } \), and \( w( \alpha ) \) is precisely the number of variables that divide \( z^{ \alpha } \).

Let \( n \in \Z^+ \). A \emph{complex polynomial function in \( n \) variables} (\emph{polynomial} for short) is a function \( P \colon \C^n \to \C \) of the form
\[
P(z) = \sum_{ \alpha } c_{ \alpha } (P) \, z^{ \alpha } \, ,
\]
with \( c_{ \alpha } (P) \in \C \) and \( c_{ \alpha } (P) \neq 0 \) for finitely many \( \alpha \). Given \( m \in \Z^+ \), the polynomial \( P \) is said to be \emph{\( m \)-homogeneous} if \( c_{ \alpha } (P) = 0 \) whenever \( \abs{ \alpha } \neq m \). We denote the vector space of \( m \)-homogeneous polynomials in \( n \) variables by \( \CP_m^{ (n) } \), and we endow it with the supremum norm \( \norm{ \ }_{ \infty } \) on the \( n \)-dimensional unit polydisc:
\[
\norm{P}_{ \infty } = \sup_{ z \in \D^n } \abs{ P(z) } \, .
\]

\section{Introduction}

\noindent The polynomial \noun{Bohnenblust--Hille} inequality asserts that, for each \( m \in \Z^+ \), there is a constant \( C_m \geq 1 \) such that for all \( n \in \Z^+ \) and every \( P \in \CP_m^{ (n) } \):
\[
\Biggl ( \, \sum_{ \alpha } \, \abs{ c_{ \alpha } (P) }^{ \frac{ 2m }{ m + 1 } } \Biggr )^{ \frac{ m + 1 }{ 2m } } \leq C_m \cdot \norm{P}_{ \infty } \, .
\]
Although it is well-known that the value \( \frac{ 2m }{ m + 1 } \) is optimal, the determination of the optimal constants and their rate of growth remains as a central question.

A natural support-restricted problem was addressed by \noun{D.~Carando}, \noun{A.~Defant} and \noun{P.~Sevilla-Peris}~\cite{CDS}. Their estimates apply, in particular, to polynomials whose monomials with nonzero coefficient involve at most \( \MW \) variables (equivalently, \( w( \alpha ) \leq \MW \)), and give constants of polynomial growth in \( m \). More precisely, their argument yields a bound of order \( m^{ ( \MW + 1 ) / 2 } \) for fixed \( \MW \).

\noun{M.~Maia}, \noun{T.~Nogueira} and \noun{D.~Pellegrino}~\cite{MNP} later obtained bounds that are uniform in \( m \) for each fixed \( \MW \). Their framework is slightly stronger: the coefficient-related sum is restricted to those \( \alpha \) with \( w( \alpha ) \leq \MW \), but the polynomial itself need not be supported on that set (that is, the coefficients \( c_{ \alpha } (P) \) with \( w ( \alpha ) > \MW \) are allowed to be nonzero). It implies, in particular, uniform boundedness for the restricted class considered in this work.

The purpose of this paper is to determine the asymptotic behavior of the corresponding optimal constants. We prove that, for every fixed \( \MW \), these constants tend to \( 1 \) as \( m \) tends to infinity. More precisely, we obtain the quantitative upper estimate \( A_{ \MW }^{ \MW / m } \cdot m^{ ( \MW^2 - 1 ) / ( 2m ) } \), where \( A_{ \MW } \geq 1 \) depends only on \( \MW \) (Theorem~\ref{thm:main}). Thus, for the restricted class considered here, the constants are not merely uniformly bounded: they are asymptotically contractive. We emphasize that, although our result concerns a more restrictive class than the one considered in~\cite{MNP}, its novelty lies in the identification of the precise limiting value of the corresponding constants.

The proof is specialized to the homogeneous problem and uses its full combinatorial structure. A monomial of degree \( m \) supported on exactly \( k \) variables determines a positive composition of \( m \), that is, a \( k \)-tuple of positive integers that sum to \( m \); there are \( \binom{ m - 1 }{ k - 1 } \) such compositions. After partitioning the variables into \( k \) labeled blocks, the monomials whose supports meet every block can be represented as the coefficients of a \( k \)-linear form, to which we apply the multilinear \noun{Bohnenblust--Hille} inequality (Proposition~\ref{prop:support-level}); this block-decomposition device is related to the one used in the support-sensitive setting of \noun{A.~Defant}, \noun{D.~Galicer}, \noun{M.~Mansilla}, \noun{M.~Masty{\l}o} and \noun{S.~Muro}~\cite{DGMMM}. A bounded projection onto exact support levels (Lemma~\ref{lem:projection}), together with an interpolation argument, implies our result.

At the end, we compare our argument with the recent support-sensitive \noun{Bohnenblust--Hille} inequality from~\cite{DGMMM}. Their result is considerably more general and already implies the qualitative limit by a standard interpolation argument. In the homogeneous setting considered here, our proof uses the constraint on the sum of the exponents which allows us to obtain a smaller power of \( m \).

From now on, we work exclusively with elements of \( \CP_m^{ (n) } \), and the letters \( k \) and \( \MW \) denote positive integers. In addition, we define
\[
q_r \coloneq \frac{ 2r }{ r + 1 } \, , \quad \textnormal{for} \ r \in \Z^+ \, .
\]
These values, together with their inverses, appear prominently as exponents along our calculations. It should be noted that we work only with complex scalars, which is essential for the asymptotic contractivity conclusion. Indeed, in the real case asymptotic contractivity already fails for \( \MW = 1 \): if \( m \) is even, then the \( m \)-homogeneous polynomial \( P = x_1^m - x_2^m \) has norm one on \( ( -1 , 1 ) \times ( -1 , 1 ) \), whereas the \( \frac{ 2m }{ m + 1 } \)-norm of the vector of coefficients of \( P \) is \( 2^{ ( m + 1 ) / ( 2m ) } \to \sqrt{2} \).

\section{The restricted optimal constants}

\noindent Let
\begin{equation} \label{defLambdaW}
\CP_{ m , \MW }^{ (n) } = \bigl \{ P \in \CP_m^{ (n) } \colon c_{ \alpha } (P) = 0 \quad \textnormal{whenever} \ w ( \alpha ) > \MW \bigr \} \, .
\end{equation}
Thus, \( \CP_{ m , \MW }^{ (n) } \) consists of those polynomials \( P \in \CP_m^{ (n) } \) whose monomials with nonzero coefficient are divisible by at most \( \MW \) variables. If \( \alpha \in \Znonneg^n \) satisfies \( \abs{ \alpha } = m \), then \( w ( \alpha ) \leq \min \{ m , n \} \). Consequently, \( \MW \geq \min \{ m , n \} \) implies \( \CP_{ m , \MW }^{ (n) } = \CP_m^{ (n) } \). Conversely, suppose that \( \CP_{ m , \MW }^{ (n) } = \CP_m^{ (n) } \). If \( m \leq n \), then \( z^{ \alpha } \coloneq z_1 \dotsm z_m \in \CP_{ m , \MW }^{ (n) } \)\,, hence \( \MW \geq w( \alpha ) = m \); similarly, if \( m > n \), then \( z^{ \alpha } \coloneq z^{ m - n + 1 } z_2 \dotsm z_n \in \CP_{ m , \MW }^{ (n) } \)\,, hence \( \MW \geq w( \alpha ) = n \). In any case, we conclude that \( \MW \geq \min \{ m , n \} \).

We define \( K_{ m , \MW } \) as the least constant such that
\begin{equation} \label{def K_m,MW}
\Biggl ( \, \sum_{ \alpha } \, \abs{ c_{ \alpha } (P) }^{ \frac{ 2m }{ m + 1 } } \Biggr )^{ \frac{ m + 1 }{ 2m } } \leq K_{ m , \MW } \cdot \norm{P}_{ \infty } \, , \quad \textnormal{for all} \ n \ \textnormal{and every} \ P \in \CP_{ m , \MW }^{ (n) } \, .
\end{equation}
A polynomial estimate for these constants was first obtained in~\cite{CDS}, namely
\[
K_{ m , \MW } \leq 2^{ \MW / 2 } \, m^{ ( \MW + 1 ) / 2 } \, .
\]
Such estimate was substantially improved in~\cite{MNP}, where the authors consider a broader formulation, in which the polynomial itself may contain monomials with arbitrary support, whereas the coefficient sum is restricted to monomials involving at most \( \MW \) variables. In particular, their estimate applies to the restricted class considered in~\eqref{defLambdaW}. If \( \widetilde{K}_{ m , \MW } \) denotes the optimal constant in that formulation, then from the proof of~\cite{MNP}*{Theorem~1.2} it follows that
\[
\limsup_{ m \to \infty } \, \widetilde{K}_{ m , \MW } \leq \MW^{ \MW } e^{ \MW }\, .
\]
In particular we have \( \widetilde{K}_{ m , \MW } \leq \kappa_{ \MW } \), where \( \kappa_{ \MW } \) depends only on \( \MW \).

The distinction between the method from~\cite{MNP} and ours is essential: our approach uses crucially the absence of support levels larger than \( \MW \), together with the homogeneity constraint. This allows to group the monomials with nonzero coefficient according to their support size and to exploit the combinatorics of the corresponding positive compositions.

Our main result is the following.

\begin{theorem} \label{thm:main}

For each \( \MW \) there exists a constant \( A_{ \MW } \geq 1 \) such that
\begin{equation} \label{main inequality}
1 \leq K_{ m , \MW } \leq A_{ \MW }^{ \MW / m } \cdot m^{ ( \MW^2 - 1 ) / ( 2m ) } \, , \quad \textnormal{for every} \ m \geq \MW \, .
\end{equation}
Consequently,
\begin{equation} \label{main limit}
\lim_{ m \to \infty } K_{ m , \MW } = 1 \quad \textnormal{for every} \ \MW \, .
\end{equation}

\end{theorem}
Thus, while~\cite{MNP} provides a uniform estimate in a more general setting, Theorem~\ref{thm:main} gives an explicit asymptotically contractive bound for the original homogeneous support-restricted problem.

The lower estimate in~\eqref{main inequality} follows by taking
\( P=z_1^m \) in~\eqref{def K_m,MW}, and~\eqref{main limit} immediately follows from the upper estimate in~\eqref{main inequality}. The remainder of the paper is devoted to the proof of that upper estimate.

\section{Bounded projections onto exact support levels}

\noindent Given \( P \in \CP_{ m , \MW }^{ (n) } \)\,, we write
\begin{equation} \label{P = P1 + + PM}
P = P_1 + \dotsb + P_{ \MW } \, , \quad \textnormal{where} \ P_j \coloneq \sum_{ w( \alpha ) = j } c_{ \alpha } (P) z^{ \alpha } \in \CP_{ m , \MW }^{ (n) } \, .
\end{equation}
We shall need a dimension-free estimate for the supremum norm of each exact support-level component \( P_k \) of \( P \). The following result provides such an estimate.

\begin{lemma} \label{lem:projection}

For each \( \MW \) and \( k \) with \( k \leq \MW \), there exists a constant \( L_{ \MW , k } \geq 1 \), depending only on \( \MW \) and \( k \), such that
\begin{equation} \label{||P_k|| < = L_{ MW , k } ||P||}
\norm{ P_k }_\infty \leq L_{ \MW , k } \cdot \norm{P}_{ \infty } \, , \quad \textnormal{for all} \ n , m \ \textnormal{and every} \ P \in \CP_{ m , \MW }^{ (n) } \, .
\end{equation}

\end{lemma}

\begin{proof}

Let \( P \in \CP_{ m , \MW }^{ (n) } \)\,. Given \( t \in [ 0 , 1 ] \), let \( \varepsilon = (\varepsilon_1 , \dotsc , \varepsilon_n) \) be an \( n \)-tuple of independent \noun{Bernoulli} random variables satisfying \( \varmathbb{P} ( \varepsilon_r = 1 ) = t \) and \( \varmathbb{P} ( \varepsilon_r = 0 ) = 1 - t \) for each \( r \), and define \( \EtP \colon \C^n \to \C \) by \( \EtP (z) = \E \bigl [ P( z_1 \varepsilon_1 , \dotsc , z_n \varepsilon_n ) \bigr ] \). By~\eqref{P = P1 + + PM} we have
\begin{align*}
\EtP (z) & = \E \Biggl [ \sum_{ j = 1 }^{ \MW } \sum_{ w( \alpha ) = j } c_{ \alpha } (P) \, z^{ \alpha } \cdot \varepsilon^{ \alpha } \Biggr ] = \sum_{ j = 1 }^{ \MW } \sum_{ w( \alpha ) = j } c_{ \alpha } (P) \, z^{ \alpha } \E [ \varepsilon^{ \alpha } ] \, .
\intertext{Given \( \alpha \in \Znonneg^n \)\,, we have \( \varepsilon_r^{ \alpha_r } = \varepsilon_r \) if \( r \in \supp ( \alpha ) \), and \( 1 \) otherwise. This, together with the independence of the \( \varepsilon_r \)\,, yields \( \E [ \varepsilon^{ \alpha } ] = \prod_{ r \in \supp ( \alpha ) } \E [ \varepsilon_r ] = t^{ w( \alpha ) } \), hence}
\EtP (z) & = \sum_{ j = 1 }^{ \MW } \sum_{ w( \alpha ) = j } c_{ \alpha } (P) \, z^{ \alpha } t^j = \sum_{ j = 1 }^{ \MW } t^j P_j (z) \, .
\end{align*}
Thus, we have
\begin{equation} \label{\EtP}
\EtP = \sum_{ j = 1 }^{ \MW } t^j P_j \in \CP_{ m , \MW }^{ (n) } \, , \quad \textnormal{for each} \ t \in [ 0 , 1 ] \, .
\end{equation}
Let \( F \) be the rational function field in \( n \) variables over \( \C \), that is, the field of fractions of the ring of polynomials in \( n \) variables over \( \C \); in particular we have \( \C \subseteq F \) and \( P_1 , \dotsc , P_{ \MW } \in F \). Let \( T \) be an indeterminate over \( F \), and consider the polynomial \( h = \sum_{ j = 1 }^{ \MW } P_j T^j \in F[T] \). Choose distinct points \( t_0 , \dotsc , t_{ \MW } \in [ 0 , 1 ] \subseteq F \), and let \( \ell_0 , \dotsc , \ell_{ \MW } \in \C [T] \subseteq F[T] \) be the corresponding \noun{Lagrange} polynomials:
\[
\ell_j = \prod_{ \substack{ 0 \leq r \leq \MW \\ r \neq j } } \frac{ T - t_r }{ t_j - t_r } \, .
\]
For \( \ell \in F[T] \), let \( [ T^k ] \ell \) be the coefficient of \( T^k \) in \( \ell \). The \noun{Lagrange} interpolation formula implies \( h = \sum_{ j = 0 }^{ \MW } h( t_j ) \ell_j \)\,. On the other hand,~\eqref{\EtP} implies \( h( t_j ) = E_{ t_j } P \). Putting these facts together yields
\[
P_k = [ T^k ] h = [ T^k ] \biggl ( \sum_{ j = 0 }^{ \MW } h( t_j ) \ell_j \biggr ) = \sum_{ j = 0 }^{ \MW } [ T^k ] \ell_j \cdot h( t_j ) = \sum_{ j = 0 }^{ \MW } [ T^k ] \ell_j \cdot E_{ t_j } P \, .
\]
Notice that \( [ T^k ] \ell_j \in \C \) for all \( j \) because \( \ell_j \in \C [T] \). Therefore we have
\[
\norm{ P_k }_{ \infty } \leq \sum_{ j = 0 }^{ \MW } \abs[\big]{ [ T^k ] \ell_j } \cdot \norm{ E_{ t_j } P }_{ \infty } \, .
\]
We claim that \( \norm{ E_{ t_j } P }_{ \infty } \leq \norm{P}_{ \infty } \) for each \( j \), and so~\eqref{||P_k|| < = L_{ MW , k } ||P||} holds with
\[
L_{ \MW , k } \coloneq \sum_{ j = 0 }^{ \MW } \abs[\big]{ [ T^k ] \ell_j } \, ,
\]
which is independent of \( m \), \( n \) and \( P \). To prove our claim, take \( t \in [ 0 , 1 ] \): given \( z \in \D^n \), we have \( ( z_1 \varepsilon_1 , \dotsc , z_n \varepsilon_n ) \in \D^n \) almost surely, hence \( \abs{ P( z_1 \varepsilon_1 , \dotsc , z_n \varepsilon_n ) } \leq \norm{P}_{ \infty } \) almost surely. Thus,
\[
\abs{ \EtP (z) } = \abs[\big]{ \E \bigl [ P( z_1 \varepsilon_1 , \dotsc , z_n \varepsilon_n ) \bigr ] } \leq \E \bigl [ \abs{ P( z_1 \varepsilon_1 , \dotsc , z_n \varepsilon_n ) } \bigr ] \leq \E [ \, \norm{P}_{ \infty } ] = \norm{P}_{ \infty } \, ,
\]
which in turn implies
\[
\norm{ \EtP }_{ \infty } = \sup_{ z \in \D^n } \abs{ \EtP (z) } \leq \norm{P}_{ \infty } \, .
\]
Finally, if for some \( m \) and \( n \) there is a nonzero \( P \in \CP_{ m , \MW }^{ (n) } \) such that \( P = P_k \)\,, then from~\eqref{||P_k|| < = L_{ MW , k } ||P||} applied to \( P \) we obtain \( L_{ \MW , k } \geq 1 \). But this is always the case: we may take, for example, \( m = n = k \), and \( P(z) = z_1 \dotsm z_k \in \CP_{ m , \MW }^{ (n) } = \CP_{ k , \MW }^{ (k) } \)\,.
\end{proof}

\section{The homogeneous support-level estimate}

\noindent Let \( B_k \geq 1 \) denote the optimal constant in the complex \( k \)-linear \noun{Bohnenblust--Hille} inequality~\cite{BH}, that is,
\begin{equation} \label{mult BH}
\Biggl ( \, \sum_{ 1 \leq i_1 , \dotsc , i_k \leq n } \abs[\big]{ A \bigl ( e^{ ( i_1 ) } , \dotsc , e^{ ( i_k ) } \bigr ) }^{ \BH_k } \Biggr )^{ \frac{1}{ \BH_k } }
 \leq B_k \cdot \norm{A} \, ,
\end{equation}
for all \( n \) and every \( k \)-linear form \( A \) on \( \C^n \). Here, for \( n \) fixed, \( e^{ (1) } , \dotsc , e^{ (n) } \) stand for the elements in the canonical basis of \( \C^n \), and \( \norm{A} \) is defined as the supremum of the values \( \abs[\big]{ A \bigl ( x^{ (1) } , \dotsc , x^{ (k) } \bigr ) } \), with \( x^{ (j) } \in \D^n \) for each \( j \). We shall use that \( B_k \) depends only on \( k \) and not on the dimensions.

For \( m \geq k \), let \( \CC_{ m , k } \) be the set of positive compositions of \( m \) into \( k \) parts, that is
\[
\CC_{ m , k } = \{ \beta = ( \beta_1 , \dotsc ,\beta_k ) \in ( \Z^+ )^k \colon \beta_1 + \dotsb + \beta_k = m \}.
\]
The proof of the following result uses certain partitions of the coordinate set \( \{ 1 , \dotsc , n \} \), which are closely related to the random block decomposition considered in~\cite{DGMMM}*{Section~2.2}.

\begin{proposition} \label{prop:support-level}

Let \( P \in \CP_m^{ (n) } \) with \( w( \alpha ) = k \) whenever \( c_{ \alpha } (P) \neq 0 \). Then
\begin{equation} \label{eq:support-level}
\Biggl ( \, \sum_{ \alpha } \, \abs{ c_{ \alpha } (P) }^{ \BH_k } \Biggr )^{ \frac{1}{ \BH_k } } \leq \biggl [ \frac{ k^k }{ k! } \, \binom{ m - 1 }{ k - 1 } \biggr ]^{ \frac{1}{ \BH_k } } B_k \cdot \norm{P}_{ \infty } \, .
\end{equation}

\end{proposition}

\begin{proof}

Recall that \( w( \alpha ) \leq \min \{ m , n \} \) whenever \( c_{ \alpha } (P) \neq 0 \). Thus, \( k > \min \{ m , n \} \) implies \( P = 0 \), and the result follows (\( \binom{ m - 1 }{ k - 1 } = 0 \) for \( k > m \), by definition). Therefore we may assume \( k \leq \min \{ m , n \} \). In what follows \( r \) always denotes an integer between \( 1 \) and \( k \).

For a map \( \sigma \colon \{ 1 , \dotsc , n \} \to \{ 1 , \dotsc , k \} \), let \( I_r [ \sigma ] = \sigma^{ -1 } ( \{ r \} ) \). Notice that \( \{ 1 , \dotsc , n \} = I_1 [ \sigma ] \uplus \dotsb \uplus I_k [ \sigma ] \) (disjoint union), but some blocks \( I_r [ \sigma ] \) may be empty.

For \( \beta = ( \beta_1 , \dotsc , \beta_k ) \in \CC_{ m , k } \)\,, we define the \emph{multihomogeneous \noun{Fourier} projection} \( \QsigbeP \colon \C^n \to \C \) by
\begin{equation} \label{defQsigbeP}
\QsigbeP (z) = \textstyle \int_{ \T^k } P( \omega_{ \sigma (1) } \, z_1 , \dotsc , \omega_{ \sigma (n) } \, z_n ) \, \overline{ \omega_1 }^{ \, \beta_1 } \dotsm \, \overline{ \omega_k }^{ \, \beta_k } \, d \omega \, ,
\end{equation}
where \( \T \) is the unit circle in \( \C \) and \( d \omega \) stands for the normalized \noun{Haar} measure on \( \T^k \).

In order to determine \( \QsigbeP \)\,, we first evaluate \( Q_{ \sigma , \beta , \widetilde{P} } \)\,, with \( \widetilde{P} (z) = z^{ \alpha } \) and \( w( \alpha ) = k \). Using that \( \{ 1 , \dotsc , n \} = I_1 [ \sigma ] \uplus \dotsb \uplus I_k [ \sigma ] \) we get
\begin{align*}
\widetilde{P} ( \omega_{ \sigma (1) } \, z_1 , \dotsc , \omega_{ \sigma (n) } \, z_n ) & = z^{ \alpha } \prod_{ j = 1 }^n \omega_{ \sigma (j) }^{ \alpha_j } = z^{ \alpha } \prod_{ r = 1 }^k \prod_{ j \in I_r [ \sigma ] } \omega_{ \sigma (j) }^{ \alpha_j } = z^{ \alpha } \prod_{ r = 1 }^k \prod_{ j \in I_r [ \sigma ] } \omega_r^{ \alpha_j } \\[3mm]
& = z^{ \alpha } \prod_{ r = 1 }^k \bigl ( \omega_r \bigr ) \mystrut^{ \textstyle \sum_{ j \in I_r [ \sigma ] } \, \alpha_j } \, ,
\end{align*}
hence
\[
Q_{ \sigma , \beta , \widetilde{P} } (z) = z^{ \alpha } \prod_{ r = 1 }^k \int_{ \T } \, \bigl ( \omega_r \bigr ) \mystrut^{ \textstyle \sum_{ j \in I_r [ \sigma ] } \, \alpha_j } \cdot \bigl ( \, \overline{ \omega_r } \, \bigr ) \mystrut^{ \, \textstyle \beta_r } \, d \omega_r \, .
\]
Given \( b , c \in \Znonneg \), we have \( \int_{ \T } \omega_r^b \cdot \overline{ \omega_r }^{ \, c } \, d \omega_r = 0 \) if \( b \neq c \), and \( 1 \) otherwise. Moreover, since \( \alpha_j = 0 \) whenever \( j \notin \supp ( \alpha ) \), we have \( \sum \limits_{ j \in I_r [ \sigma ] } \alpha_j \ = \sum \limits_{ j \in I_r [ \sigma ] \cap \supp ( \alpha ) } \alpha_j \)\,.

\medskip

The calculations above, together with \( \C \)-linearity of the integral, show that
\begin{equation} \label{P implies Q}
\QsigbeP (z) = \sum_{ \alpha \in \Ssigbe } c_{ \alpha } (P) \, z^{ \alpha } \, ,
\end{equation}
where \( \Ssigbe \) is the set of those \( \alpha \) satisfying \( \abs{ \alpha } = m \), \( w( \alpha ) = k \), and the additional conditions
\begin{equation} \label{sum alpha_j = beta_r}
\sum_{ \mathclap{ j \in I_r [ \sigma ] \cap \supp ( \alpha ) } } \alpha_j = \beta_r \, , \quad \textnormal{for each} \ r \, .
\end{equation}
This shows that \( \QsigbeP \in \CP_m^{ (n) } \). Notice that \( \Ssigbe \) does not depend on \( P \), but only on \( \sigma \) and \( \beta \) (recall that \( m \) and \( k \) are fixed in our reasoning).

Let \( I[ \sigma ] \coloneq I_1 [ \sigma ] \times \dotsb \times I_k [ \sigma ] \). For \( s = ( s_1 , \dotsc , s_k ) \in I[ \sigma ] \), let \( \alpha [ s , \beta ] \in \Znonneg^n \) be given by
\begin{equation} \label{alphasbeta}
\alpha [ s , \beta ]_j =
\begin{cases}
\beta_r \, , & \ \textnormal{for} \ j = s_r \, ; \\[3mm]
0 \, , & \ \textnormal{otherwise} \, .
\end{cases}
\end{equation}
We claim that \( S_{ \sigma , \beta } \) is precisely the set of tuples \( \alpha[ s , \beta ] \) with \( s \in I[ \sigma ] \), and so~\eqref{P implies Q} and~\eqref{alphasbeta} imply
\[
\QsigbeP (z) = \sum_{ s \in I[ \sigma ] } c_{ \alpha [ s , \beta ] } (P) \ z_{ s_1 }^{ \beta_1 } \dotsm z_{ s_k }^{ \beta_k } \, .
\]
To prove this fact, notice that \( \supp ( \alpha [ s , \beta ] ) = \{ s_1 , \dotsc , s_k \} \), hence \( w( \alpha ) = k \). Moreover, \( \abs{ \, \alpha [ s , \beta ] \, } = \alpha [ s , \beta ]_{ s_1 } + \dotsb + \alpha [ s , \beta ]_{ s_k } = \beta_1 + \dotsb + \beta_k = m \), and \( I_r [ \sigma ] \cap \supp ( \alpha [ s , \beta ] ) = \{ s_r \} \) for each \( r \), hence \( \alpha [ s , \beta ] \) satisfies~\eqref{sum alpha_j = beta_r}, by~\eqref{alphasbeta}. Conversely, let \( \alpha \in \Ssigbe \)\,. Since \( \beta_r > 0 \),~\eqref{sum alpha_j = beta_r} implies \( \alpha_j > 0 \) for some \( j \in I_r [ \sigma ] \cap \supp ( \alpha ) \). Thus, the \( k \)-element set \( \supp ( \alpha ) \) contains one element from each one of the \( k \) pairwise disjoint blocks \( I_r [ \sigma ] \), hence \( \supp ( \alpha ) = \{ s_1 , \dotsc , s_k \} \), with \( s_r \in I_r [ \sigma ] \) for each \( r \). If \( s = ( s_1 , \dotsc , s_k ) \), then \( s \in I[ \sigma ] \), and~\eqref{sum alpha_j = beta_r}, \eqref{alphasbeta} imply \( \alpha = \alpha [ s , \beta ] \).

Associate with \( \QsigbeP \) the \( k \)-linear form \( \AsigbeP \) on \( \C^n \) given by
\[
\AsigbeP \bigl ( x^{ (1) } , \dotsc , x^{ (k) } \bigr ) = \sum_{ s \in I[ \sigma ] } c_{ \alpha [ s , \beta ] } (P) \ x_{ s_1 }^{ (1) } \dotsb x_{ s_k }^{ (k) } \, .
\]
Given \( x^{ (1) } , \dotsc , x^{ (k) } \in \D^n \) we pick, for each \( r \) and each \( j \in I_r [ \sigma ] \)\,, a \( \beta_r \)-th root \( z_j \) of \( x_j^{ (r) } \). We have \( z_j \in \D \) because \( x_j^{ (r) } \in \D \), and \( \AsigbeP \bigl ( x^{ (1) } , \dotsc , x^{ (k) } \bigr ) = \QsigbeP (z) \). On the other hand, for \( z \in \D^n \) and each \( r \) the elements \( x^{ (r) } = ( z_1^{ \beta_r } , \dotsc , z_n^{ \beta_r } ) \) are in \( \D^n \) and satisfy \( \QsigbeP (z) = \AsigbeP \bigl ( x^{ (1) } , \dotsc , x^{ (k) } \bigr ) \). From this, we conclude that \( \norm{ \AsigbeP } = \norm{ \QsigbeP }_{ \infty } \)\,.

We claim that \( \norm{ \QsigbeP }_{ \infty } \leq \norm{P}_{ \infty } \)\,, hence \( \norm{ \AsigbeP } \leq \norm{P}_{ \infty } \)\,. This, together with the multilinear \noun{Bohnenblust--Hille} inequality~\eqref{mult BH} applied to \( \AsigbeP \)\,, gives
\begin{equation} \label{sum_s < = B ||P||}
\sum_{ s \in I[ \sigma ] } \abs[\big]{ c_{ \alpha [ s , \beta ] } (P) }^{ \BH_k } \leq B_k^{ \BH_k } \cdot \norm{P}_{ \infty }^{ \BH_k } \, .
\end{equation}
To prove our claim, take \( \omega_1 , \dotsc , \omega_k \in \T \) and \( z \in \D^n \). We have \( ( \omega_{ \sigma (1) } \, z_1 , \dotsc , \omega_{ \sigma (n) } \, z_n ) \in \D^n \), hence \( \abs{ P( \omega_{ \sigma (1) } \, z_1 , \dotsc , \omega_{ \sigma (n) } \, z_n ) } \leq \norm{P}_{ \infty } \)\,. Thus,~\eqref{defQsigbeP} implies
\begin{align*}
\abs{ \QsigbeP (z) } & \leq \textstyle \int_{ \T^k } \abs[\big]{ P( \omega_{ \sigma (1) } \, z_1 , \dotsc , \omega_{ \sigma (n) } \, z_n ) \, \overline{ \omega_1 }^{ \, \beta_1 } \dotsm \, \overline{ \omega_k }^{ \, \beta_k } } \, d \omega \\[3mm]
& \textstyle \leq \int_{ \T^k } \norm{P}_{ \infty } \, d \omega = \norm{P}_{ \infty } \, ,
\end{align*}
from which the inequality \( \norm{ \QsigbeP }_{ \infty } \leq \norm{P}_{ \infty } \) follows.

Let \( S \) be the sum of the left side of~\eqref{sum_s < = B ||P||} over all the pairs \( ( \sigma , \beta ) \). Since there are \( k^n \) maps \( \sigma \colon \{ 1 , \dotsc , n \} \to \{ 1 , \dotsc , k \} \) and \( \CC_{ m , k } \) has \( \binom{ m - 1 }{ k - 1 } \) elements, we have
\[
S = \sum_{ \sigma , \beta } \sum_{ s \in I[ \sigma ] }
\abs[\big]{ c_{ \alpha [ s , \beta ] } (P) }^{ \BH_k } \leq k^n \binom{ m - 1 }{ k - 1 } B_k^{ \BH_k } \cdot \norm{P}_{ \infty }^{ \BH_k } \, .
\]
Finally, for each \( \alpha \) with \( \abs{ \alpha } = m \) and \( w ( \alpha ) = k \), we claim that \( \alpha = \alpha [ s , \beta ] \) for some \( ( \sigma , \beta ) \) and some \( s \in I[ \sigma ] \) if, and only if,
\[
\sigma [ \alpha ] \coloneq \sigma \mathord{ \restriction }_{ \supp ( \alpha ) } \colon \supp ( \alpha ) \to \{ 1 , \dotsc , k \}
\]
is a bijection from \( \supp ( \alpha ) \) onto \( \{ 1 , \dotsc , k \} \) and \( \beta_r = \alpha_{ \sigma [ \alpha ]^{ -1 } (r) } \) for each \( r \). The first condition is satisfied by exactly \( k! \cdot k^{ n - k } \) maps \( \sigma \), and each such map determines uniquely the composition \( \beta \). Therefore the summand \( \abs{ c_{ \alpha } (P) }^{ \BH_k } \) appears exactly \( k! \cdot k^{ n - k } \) times in the iterated sum that defines \( S \), hence
\[
\sum_{ \alpha } \abs{ c_{ \alpha } (P) }^{ \BH_k } = \frac{S}{ k! \cdot k^{ n - k } } \leq \frac{ k^n }{ k! \cdot k^{ n - k } } \binom{ m - 1 }{ k - 1 } B_k^{ \BH_k } \cdot \norm{P}_{ \infty }^{ \BH_k } \, ,
\]
from which~\eqref{eq:support-level} follows.

Regarding the proof of the remaining assertion, if \( \alpha = \alpha [ s , \beta ] \), with \( s = ( s_1 , \dotsc , s_k ) \in I[ \sigma ] \), then \( \sigma ( s_r ) = r \) for each \( r \), and since \( \supp ( \alpha ) = \{ s_1 , \dotsc , s_k \} \) by~\eqref{alphasbeta}, we conclude that \( \sigma [ \alpha ] \) is a bijection from \( \supp ( \alpha ) \) onto \( \{ 1 , \dotsc , k \} \), and for each \( r \) we have \( \sigma [ \alpha ]^{ -1 } (r) = s_r \)\,, hence \( \alpha_{ \sigma [ \alpha ]^{ -1 } (r) } = \alpha_{ s_r } = \beta_r \)\,, again by~\eqref{alphasbeta}. Conversely, suppose that \( \sigma [ \alpha ] \) is a bijection from \( \supp ( \alpha ) \) onto \( \{ 1 , \dotsc , k \} \) and \( \beta_r = \alpha_{ \sigma [ \alpha ]^{ -1 } (r) } \) for each \( r \). If \( s_r = \sigma [ \alpha ]^{ -1 } (r) \), then \( \supp ( \alpha ) = \{ s_1 , \dotsc , s_k \} \), and \( s = ( s_1 , \dotsc , s_k ) \in I[ \sigma ] \) satisfies \( \alpha = \alpha [ s , \beta ] \): in fact, by~\eqref{alphasbeta} we have \( \alpha [ s , \beta ]_{ s_r } = \beta_r = \alpha_{ s_r } \)\,, and for \( j \neq s_1 , \dotsc , s_k \) we have \( \alpha [ s , \beta ]_j = 0 \), whereas \( \alpha_j = 0 \) because \( j \notin \supp ( \alpha ) \).
\end{proof}

\noindent The gain over an unrestricted support-sensitive estimate comes from the number \( \binom{ m - 1 }{ k - 1 } \) of positive compositions. In a nonhomogeneous problem the exponents in the \( k \) active coordinates may vary independently, whereas homogeneity forces their sum to be equal to \( m \).

\section{A polynomial estimate at the support exponent}

\noindent We now combine the exact-level estimate with the bounded projections from Lemma~\ref{lem:projection}.

\begin{theorem} \label{thm:support-exponent}

For each \( \MW \in \Z^+ \) there exists \( A_{ \MW } \geq 1 \) such that, for all \( m , n \) and every \( P \in \CP_{ m , \MW }^{ (n) } \)\,, one has
\begin{equation} \label{eq:support-exponent}
\Biggl ( \, \sum_{ \alpha } \, \abs{ c_{ \alpha } (P) }^{ \BH_{ \MW } } \Biggr )^{ \frac{1}{ \BH_{ \MW } } } \leq A_{ \MW } \cdot m^{ ( \MW^2 - 1 ) /( 2 \MW ) } \cdot \norm{P}_{ \infty } \, .
\end{equation}

\end{theorem}

\begin{proof}

Write \( P = \sum_{ j = 1 }^{ \MW } P_j \) as in~\eqref{P = P1 + + PM}. For \( k \leq \MW \) we have \( \BH_k \leq \BH_{ \MW } \)\,, and so the monotonicity of finite-dimensional \( \ell_p \)-norms implies
\begin{equation} \label{sumqW<=sumqk}
\Biggl ( \, \sum_{ w( \alpha ) = k } \abs{ c_{ \alpha } (P) }^{ \BH_{ \MW } } \Biggr )^{ \frac{1}{ \BH_{ \MW } } } \leq \Biggl ( \, \sum_{ w( \alpha ) = k } \abs{ c_{ \alpha } (P) }^{ \BH_k } \Biggr )^{ \frac{1}{ \BH_k } } \, .
\end{equation}
From the definition of \( P_k \) we get
\begin{equation} \label{sumPwalpha=k}
\Biggl ( \, \sum_{ w( \alpha ) = k } \abs{ c_{ \alpha } (P) }^{ \BH_k } \Biggr )^{ \frac{1}{ \BH_k } } = \Biggl ( \, \sum_{ \alpha } \, \abs{ c_{ \alpha } ( P_k ) }^{ \BH_k } \Biggr )^{ \frac{1}{ \BH_k } } \, .
\end{equation}
Applying Proposition~\ref{prop:support-level} to \( P_k \)\,, together with the inequality \( \binom{ m - 1 }{ k - 1 } \leq \frac{ m^{ k - 1 } }{ ( k - 1 )! } \)\,, yields
\begin{equation} \label{sumPk<=Bk||Pk||}
\Biggl ( \, \sum_{ \alpha } \, \abs{ c_{ \alpha } ( P_k ) }^{ \BH_k } \Biggr )^{ \frac{1}{ \BH_k } } \leq \biggl [ \frac{ k^k }{ k! } \, \frac{ m^{ k - 1 } }{ ( k - 1 )! } \biggr ]^{ \frac{1}{ \BH_k } } B_k \cdot \norm{ P_k }_{ \infty } \, .
\end{equation}
Finally, we have \( \norm{ P_k }_{ \infty } \leq L_{ \MW , k } \cdot \norm{P}_{ \infty } \) by Lemma~\ref{lem:projection}. From this, together with~\eqref{sumqW<=sumqk}--\eqref{sumPk<=Bk||Pk||}, we obtain
\begin{equation} \label{algo < = D_{ MW },k}
\Biggl ( \, \sum_{ w( \alpha ) = k } \abs{ c_{ \alpha } (P) }^{ \BH_{ \MW } } \Biggr )^{ \frac{1}{ \BH_{ \MW } } } \leq D_{ \MW , k } \cdot m^{ ( k - 1 ) / \BH_k } \cdot \norm{P}_{ \infty } \, ,
\end{equation}
where
\[
D_{ \MW , k } \coloneq L_{ \MW , k } \cdot B_k \cdot \biggl [ \frac{ k^k }{ k! ( k - 1 )! } \biggr ]^{ \frac{1}{ \BH_k } } \, .
\]
We have \( \frac{ k - 1 }{ \BH_k } = \frac{ ( k - 1 ) ( k + 1 ) }{ 2k } = \frac{ k^2 - 1 }{ 2k } \leq \frac{ \MW^2 - 1 }{ 2 \MW} \)\,, and therefore we may replace \( m^{ ( k - 1 ) / \BH_k } \) with \( m^{ ( \MW^2 - 1 ) / ( 2 \MW ) } \) on the right side of~\eqref{algo < = D_{ MW },k}. Having this, we may take \( \BH_{ \MW} \)-th powers on both sides, which yields (recall that \( P \in \CP_{ m , \MW }^{ (n) } \))
\begin{align*}
\sum_{ \alpha } \, \abs{ c_{ \alpha } (P) }^{ \BH_{ \MW } } & = \sum_{ k = 1 }^{ \MW } \biggl ( \, \sum_{ w( \alpha ) = k } \abs{ c_{ \alpha } (P) }^{ \BH_{ \MW } } \biggr ) \\[3mm]
& \leq \biggl ( \, \sum_{ k = 1 }^{ \MW } D_{ \MW , k }^{ \BH_{ \MW } } \biggr ) \bigl [ m^{ ( \MW^2 - 1 ) / ( 2 \MW ) } \cdot \norm{P}_{ \infty } \bigr ]^{ \BH_{ \MW } } \, ,
\end{align*}
and therefore
\[
\Biggl ( \, \sum_{ \alpha } \, \abs{ c_{ \alpha } (P) }^{ \BH_{ \MW } } \Biggr )^{ \frac{1}{ \BH_{ \MW } } } \leq A_{ \MW } \cdot m^{ ( \MW^2 - 1 ) / ( 2 \MW ) } \cdot \norm{P}_{ \infty } \, ,
\]
where
\[
A_{ \MW } \coloneq \biggl ( \, \sum_{ k = 1 }^{ \MW } D_{ \MW , k }^{ \BH_{ \MW } } \biggr )^{ \frac{1}{ \BH_{ \MW } } } \, .
\]
Finally, we have \( A_{ \MW }^{ \BH_{ \MW } } \geq D_{ \MW , 1 }^{ \BH_{ \MW } } = ( L_{ \MW , 1 } \cdot B_1 )^{ \BH_{ \MW } } \). Since \( L_{ \MW , 1 } \geq 1 \) by Lemma~\ref{lem:projection} and \( B_k \geq 1 \) for all \( k \) (we actually have \( B_1 = 1 \)), we conclude that \( A_{ \MW } \geq L_{ \MW , 1 } \cdot B_1 \geq 1 \).
\end{proof}

\begin{remark}

For fixed \( \MW \), the degree dependence in~\eqref{eq:support-exponent} is
\[
m^{ ( \MW^2 - 1 ) / ( 2 \MW ) } = \frac{ m^{ \MW / 2 } }{ m^{ 1 / ( 2 \MW ) } } \, .
\]
The saving of \( \frac{1}{ 2 \MW } \) in the power of \( m \) is precisely the contribution of the homogeneous constraint \( \beta_1 + \dotsb + \beta_{ \MW } = m \).

\end{remark}

\section{Proof of asymptotic contractivity}

\noindent We are now ready to prove the main result.

\begin{proof}[Proof of Theorem~\ref{thm:main}]

We must show that
\[
\Biggl ( \, \sum_{ \alpha } \, \abs{ c_{ \alpha } (P) }^{ \frac{ 2m }{ m + 1 } } \Biggr )^{ \frac{ m + 1 }{ 2m } } \leq A_{ \MW }^{ \MW / m } \cdot m^{ ( \MW^2 - 1 ) / ( 2m ) } \cdot \norm{P}_{ \infty }
\]
for all \( m , n \) with \( m \geq \MW \) and every nonzero \( P \in \CP_{ m , \MW }^{ (n) } \)\,. Let \( x \) be the vector of nonzero coefficients of \( P \) (in any order). By \noun{Parseval}'s identity on \( \T^n \),
\begin{equation} \label{||x||_2 < = ||P||}
\norm{x}_2 = \norm{P}_{ L^2 ( \T^n ) } \leq \norm{P}_{ \infty } \, .
\end{equation}
Since \( 0 \leq \frac{ \MW }{m} \leq 1 \) and \( \frac{ m + 1 }{ 2m } = \frac{ \MW }{m} \cdot \frac{1}{ \BH_{ \MW } } + \bigl ( 1 - \frac{ \MW }{m} \bigr ) \cdot \frac{1}{2} \)\,, the log-convexity of \( \ell_p \)-norms yields
\begin{equation} \label{||x||<=W}
\norm{x}_{ \frac{ 2m }{ m + 1 } } \leq \norm{x}_{ \BH_{ \MW } }^{ \MW / m } \cdot \norm{x}_2^{ 1 - \MW / m } = \Biggl [ \, \Biggl ( \, \sum_{ \alpha } \, \abs{ c_{ \alpha } (P) }^{ \BH_{ \MW } } \Biggr )^{ \frac{1}{ \BH_{ \MW } } } \Biggr ]^{ \frac{ \MW }{m} } \cdot \norm{x}_2^{ 1 - \MW / m } \, .
\end{equation}
We may bound the factors on the right side of~\eqref{||x||<=W} via Theorem~\ref{thm:support-exponent} and~\eqref{||x||_2 < = ||P||}, respectively, obtaining
\begin{align*}
\Biggl ( \, \sum_{ \alpha } \, \abs{ c_{ \alpha } (P) }^{ \frac{ 2m }{ m + 1 } } \Biggr )^{ \frac{ m + 1 }{ 2m } } & = \norm{x}_{ \frac{ 2m }{ m + 1 } } \leq \bigl ( A_{ \MW } \cdot m^{ ( \MW^2 - 1 ) / ( 2 \MW ) } \cdot \norm{P}_{ \infty } \bigr )^{ \MW / m } \cdot \norm{P}_{ \infty }^{ 1 - \MW / m } \\[3mm]
& = A_{ \MW }^{ \MW / m } \cdot m^{ ( \MW^2 - 1 ) / ( 2m ) } \cdot \norm{P}_{ \infty } \, . \qedhere
\end{align*}

\end{proof}

\begin{remark}

When \( \MW = 1 \), every \( P \in \CP_{ m , \MW }^{ (n) } \) has the form
\[
P(z) = \sum_{j = 1 }^n a_j z_j^m \, .
\]
Let \( a = ( a_1 , \dotsc , a_n ) \). For each \( j \), let \( \abs{ a_j } = a_j \omega_j \)\,, with \( \omega_j \in \T \), and let \( z_j \in \T \) be an \( m \)-root of \( \omega_j \)\,. Then for all \( r \in ( 0 , 1 ) \) we have \( rz \in \D^n \). These facts, together with the monotonicity of finite-dimensional \( \ell_p \)-norms, imply (recall that \( \BH_m \geq 1 \))
\[
r^m \norm{a}_{ \BH_m } \leq r^m \norm{a}_1 = \sum_{ j = 1 }^n a_j ( r z_j )^m = \abs{ P ( rz ) } \leq \norm{P}_{ \infty } \, .
\]
Taking \( r \to 1^- \) yields
\[
\biggl ( \, \sum_j \, \abs{ a_j }^{ \frac{ 2m }{ m + 1 } } \biggr)^{ \frac{ m + 1 }{ 2m } } \leq \norm{P}_{ \infty } \, ,
\]
hence \( K_{ m , 1 } \leq 1 \), by~\eqref{def K_m,MW}. Thus, \( K_{ m , 1 } = 1 \) for all \( m \).

\end{remark}

\section{Comparison with a recent support-sensitive inequality}

\noindent \noun{A.~Defant}, \noun{D.~Galicer}, \noun{M.~Mansilla}, \noun{M.~Masty{\l}o} and \noun{S.~Muro} recently proved a general support-sensitive \noun{Bohnenblust--Hille} inequality on the torus~\cite{DGMMM}*{Theorem~2.12}, which constitutes a central tool in their study of low-level spaces on \noun{Hamming} schemes, and deals with a broader nonhomogeneous setting. It asserts that, for every integer \( r \geq 2 \), there exists \( c_r = O \bigl ( \sqrt{r} \, \bigr ) \) such that
\begin{equation} \label{general BH torus}
\Biggl ( \, \sum_{ \substack{ \alpha \in \{ 0 , \dotsc , r - 1 \}^n \\ w( \alpha ) \leq d } } \abs{ a_{ \alpha } }^{ \BH_d } \Biggr )^{ \frac{1}{ \BH_d } } \leq c_r^d \cdot \norm{Q}_\infty
\end{equation}
for every polynomial
\[
Q(z) = \sum_{ \substack{ \alpha \in \{ 0 , \dotsc , r - 1 \}^n \\ w( \alpha ) \leq d } } a_{ \alpha } z^{ \alpha } \, .
\]
For an admissible nonzero \( m \)-homogeneous polynomial in our problem, every coordinate exponent is at most \( m \). Moreover, by the maximum modulus principle, its supremum norm on \( \T^n \) agrees with its supremum norm on \( \D^n \). Thus~\eqref{general BH torus} can be applied with \( r = m + 1 \) and \( d = \MW \), which gives
\begin{equation} \label{||x||_pM < =}
\norm{x}_{ \BH_{ \MW } } \leq c_{ m + 1 }^{ \MW } \cdot \norm{P}_{ \infty } \, ,
\end{equation}
where \( x \) is the vector of nonzero coefficients of the polynomial. Interpolating~\eqref{||x||_pM < =} with \noun{Parseval}'s identity as in the proof of Theorem~\ref{thm:main} (see~\eqref{||x||_2 < = ||P||}, \eqref{||x||<=W}), we obtain
\[
K_{ m , \MW } \leq c_{ m + 1 }^{ \MW^2 / m } \, .
\]
Since \( c_r \leq C \sqrt{r} \) for a universal constant \( C \),
\[
K_{ m , \MW } \leq C^{ \MW^2 / m } \cdot ( m + 1 )^{ \MW^2 / ( 2m ) } \, ,
\]
and this also implies \( K_{ m , \MW } \to 1 \) for fixed \( \MW \).

The two approaches have different scopes. The estimate~\eqref{general BH torus} applies without the homogeneity constraint and is controlled by the maximum number of variables occurring in some monomial with nonzero coefficient, rather than by the total degree. Consequently, it yields an estimate of order \( ( m + 1 )^{ \MW^2 / 2m } \). On the other hand, our argument is restricted to the original homogeneous setting, but it uses the relation among the active exponents and yields the smaller exponent \( ( \MW^2 - 1 ) / ( 2m ) \).

The comparison just made concerns the polynomial dependence on \( m \) for fixed \( \MW \); the multiplicative constants have different dependence on \( \MW \) and arise from considerations of different nature.

\end{document}